\documentclass[a4paper,10pt]{article}

\usepackage{amsmath}
\usepackage{amssymb}
\usepackage{amsthm}
\usepackage{mathrsfs}
\usepackage{color}
\usepackage[dvips]{graphicx}
\usepackage[all]{xy}

\theoremstyle{plain}
\newtheorem{them}{Theorem}[section]
\newtheorem{lemma}[them]{Lemma}
\newtheorem{prop}[them]{Proposition}
\newtheorem{coro}[them]{Corollary}
\theoremstyle{definition}
\newtheorem{defi}[them]{Definition}
\newtheorem{exam}[them]{Example}
\newtheorem{rema}[them]{Remark}
\newtheorem{conv}[them]{Convention}

\newtheorem*{mthem}{Main Theorem}

\newcommand{\lmap}[3]{#1:#2 \longrightarrow #3}
\newcommand{\map}[3]{#1:#2 \rightarrow #3}

\newcommand{\nn}[2]{\overline{N_{#1}}(\mathcal{#2})}
\newcommand{\Hom}{\mathrm{Hom}}
\newcommand{\colim}{\mathop{\mathrm{colim}}}

\begin{document}

\title{The Euler characteristics of categories and the barycentric subdivision}
\author{Kazunori Noguchi \thanks{noguchi@math.shinshu-u.ac.jp}}
\date{}
\maketitle
\begin{abstract}
 We prove the $L^2$-Euler characteristic of small categories introduced by \cite{FLS} is invariant under the barycentric subdivision only for finite acyclic categories. We also extend the definition of the $L^2$-Euler characteristic and prove our extended $L^2$-Euler characteristic is invariant under the barycentric subdivision for a wider class of finite categories. 
\end{abstract}

\tableofcontents


\thispagestyle{empty}

\section{Introduction}

Euler characteristics are defined for many mathematical objects, for example, cell complexes, manifolds, varieties, graphs, and so on. But the most basic one is the Euler characteristic for simplicial complexes which is defined by the alternating sum of the number of faces. Rota defined the Euler characteristic for finite posets \cite{Rot64}. The relation between the Euler characteristic of simplicial complexes and the one of posets is described by the following diagram
$$
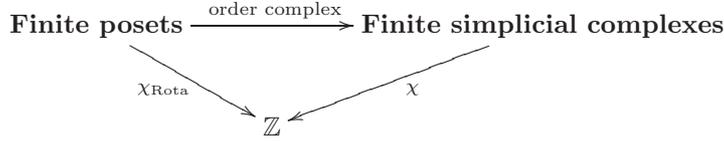
\xymatrix{
{\displaystyle \textbf{Finite posets} } \ar[dr]_{\chi_{\text{Rota}}} \ar[rr]^(0.40){\text{order complex}}&& { \displaystyle \textbf{Finite simplicial complexes} }\ar[dl]^{\chi}\\
&\mathbb{Z}
}$$
Here, the order complex of a finite poset $P$ is an abstract simplicial complex having totally ordered $(n+1)$-subsets of $P$ as its $n$-simplices.

Leinster extended Rota's theory. He defined the Euler characteristic $\chi_L$ for finite categories which satisfy certain conditions, including finite posets, finite groups, orbifolds, directed graphs and so on  \cite{Leia}. At present, we have various invariants of categories, the \textit{the series Euler characteristic} $\chi_{\sum}$ \cite{Leib}, \textit{the $L^2$-Euler characteristic} $\chi^{(2)}$ \cite{FLS}, \textit{the $L^2$-Betti numbers of discrete measured groupoids} \cite{Sau}, \textit{the Euler characteristic of $\mathbb{N}$-filtered acyclic category} $\chi_{\mathrm{fil}}$ \cite{Nog11}, \textit{the cardinality of categories} \cite{BD} and so on. In this paper, we investigate the four Euler characteristics of categories from the view point of the barycentric subdivision of categories.

First of all, let us review the four Euler characteristics of categories.

Leinster's Euler characteristic $\chi_L$ and the series Euler characteristic $\chi_{\sum}$ are defined for finite categories satisfying certain conditions. When a finite category $\mathcal{I}$ has a \textit{M\"obius inversion}, they coincide $\chi_L(\mathcal{I})=\chi_{\sum}(\mathcal{I})$. Here, a finite category $\mathcal{I}$ equipped with the set of objects $\mathrm{Ob}(\mathcal{I})=\{x_1,\dots,x_n\}$ has a M\"obius inversion if the matrix $Z_{\mathcal{I}}=(\# \Hom_{\mathcal{I}}(x_i,x_j))_{i,j}$ has the inverse matrix. But in the out of this class, all the cases occur: these two Euler characteristics take same values, they take different values, one is defined but the other is not, both of them are not defined.

The $L^2$-Euler characteristic is defined not only for finite categories but also infinite categories satisfying a certain condition. For a finite, free, skeletal EI-category $\mathcal{I}$, Leinster's Euler characteristic and the $L^2$-Euler characteristic coincide, $\chi_L(\mathcal{I})=\chi^{(2)}(\mathcal{I})$. Here, an \textit{EI-category} is a small category whose endomorphisms are isomorphisms. A small category $\mathcal{J}$ is \textit{free} if the left $\mathrm{Aut}(y)$-action on $\Hom_{\mathcal{J}}(x,y)$ is free for any objects $x,y$ of $\mathcal{J}$. Since the $L^2$-Euler characteristic is suitable for EI-condition, for a finite category $\mathcal{I}$ which is not an EI-category, it often takes different values from $\chi_L(\mathcal{I})$ and $\chi_{\sum}(\mathcal{I})$. For instance, let $M=\{0,1\}$ be the commutative monoid where 0 is the unit element and $1+1=1$. Then, we obtain $\chi^{(2)}(M)=0$ and $\chi_L(M)=\chi_{\sum}(M)=\frac{1}{2}$.

$\chi_{\mathrm{fil}}$ is the Euler characteristic for $\mathbb{N}$-filtered acyclic category. An \textit{acyclic category} is a small category whose endomorphisms and invertible morphisms are only identity morphisms. An \textit{$\mathbb{N}$-filtered acyclic category} is a pair $(\mathcal{A},\mu)$ of an acyclic category $\mathcal{A}$ and a filtration $\mu$, called an \textit{$\mathbb{N}$-filtration}, on the set of objects in $\mathcal{A}$. For a finite acyclic category $\mathcal{A}$, these four Euler characteristics coincide
$$\chi_{L}(\mathcal{A})=\chi_{\sum}(\mathcal{A})=\chi^{(2)}(\mathcal{A})=\chi_{\mathrm{fil}}(\mathcal{A},\mu)$$
for any $\mathbb{N}$-filtration $\mu$ of $\mathcal{A}$.

Moreover, $\chi_{\mathrm{fil}}$ is suitable for the barycentric subdivision of small categories. The \textit{barycentric subdivision} of small categories is a functor from the category of small categories to itself $$\lmap{\mathrm{Sd}}{\textbf{Small categories}}{\textbf{Small categories}}.$$ For a small category $\mathcal{J}$, $\mathrm{Sd}(\mathcal{J})$ is an acyclic category and its objects are the non-degenerate chains of morphisms of $\mathcal{J}$. In addition, $\mathrm{Sd}(\mathcal{J})$ has naturally an $\mathbb{N}$-filtration. Since the Euler characteristic of simplicial complexes is invariant under the barycentric subdivision, we expect a categorical analogue of this fact would hold for a certain class of small categories. But we have to note that $\mathrm{Sd}(\mathcal{J})$ is often infinite even if $\mathcal{J}$ is finite. $\mathrm{Sd}(\mathcal{J})$ is finite if and only if $\mathcal{J}$ is a finite acyclic category. So we can not always use Leinster's Euler characteristic and the series Euler characteristic for this purpose. In \cite{Nog11}, the following theorem was proven.
 
\begin{them}
Let $\mathcal{I}$ be a finite category for which the series Euler characteristic can be defined. Then, $\chi_{\mathrm{fil}}(\mathrm{Sd}(\mathcal{I}), L)$ is also defined and they coincide
$$\chi_{\Sigma}(\mathcal{I})=\chi_{\mathrm{fil}}(\mathrm{Sd}(\mathcal{I}), L),$$
that is, we have the following commutative diagram
$$\xymatrix{
 \chi_{\sum}\text{-}\mathbf{categories} \ar[dr]_{\chi_{\sum}} \ar[rr]^(0.45){\mathrm{Sd}}&&  \chi_{\mathrm{fil}}\text{-}\mathbf{categories}   \ar[dl]^{\chi_{\mathrm{fil}}}\\
&\mathbb{Q},
}$$
where $\chi_{\sum} \text{-}\mathbf{categories}$ denotes the category of finite categories for which the series Euler  characteristic can be defined and $\chi_{\text{fil}}\text{-}\mathbf{categories}$ denotes the category of $\mathbb{N}$-filtered acyclic categories for which its Euler characteristic can be defined. 
\end{them}Here, $L$ is the $\mathbb{N}$-filtration of $\mathrm{Sd}(\mathcal{I})$ which is defined by taking the length of chains.

Since the $L^2$-Euler characteristic is defined for a certain  class of infinite categories, we can consider a similar problem; is the $L^2$-Euler characteristic invariant under the barycentric subdivision? In this paper, the following theorem is obtained.
\begin{them}
For a small category $\mathcal{I}$, $\chi^{(2)}(\mathrm{Sd}(\mathcal{I})^{\mathrm{op}})$ exists if and only if $\mathcal{I}$ is finite acyclic, in which case, we have $$\chi^{(2)}(\mathrm{Sd}(\mathcal{I})^{\mathrm{op}})=\chi^{(2)}(\mathcal{I}).$$
\end{them}

Thus, the $L^2$-Euler characteristic is invariant under the barycentric subdivision only for finite acyclic  categories. But $\mathrm{Sd}(\mathcal{A})$ is a finite category for a finite acyclic category $\mathcal{A}$ and $\chi_L(\mathrm{Sd}(\mathcal{A}))$ and $\chi_{\sum}(\mathrm{Sd}(\mathcal{A}))$ exist. Furthermore, we obtain 
$$\chi_L(\mathcal{A})=\chi_L(\mathrm{Sd}(\mathcal{A})), \ \chi_{\sum}(\mathcal{A})=\chi_{\sum}(\mathrm{Sd}(\mathcal{A})).$$
And for any $\mathbb{N}$-filtration $\mu$ of $\mathcal{A}$ we obtain
$$\chi_{\mathrm{fil}}(\mathcal{A},\mu)=\chi_{\mathrm{fil}}(\mathrm{Sd}(\mathcal{A}), L ).$$ The $L^2$-Euler characteristic is defined for certain class of infinite categories but it is not suitable for the categories of the form $\mathrm{Sd}(\mathcal{I})$ since it is invariant under the barycentric subdivision only for finite acyclic categories and the operation Sd makes the von Neumann dimension trivial and generates many objects.

We introduce an extension the $L^2$-Euler characteristic, called \textit{the extended $L^2$-Euler characteristic} $\chi^{(2)}_{\mathrm{ex}}$, which is suitable for the categories after taking the functor $\mathrm{Sd}$. Then, we obtain the following theorem.

\begin{mthem}
Suppose $\mathcal{I}$ is a finite category. Then, the extended $L^2$-Euler characteristic  $\chi^{(2)}_{\mathrm{ex}}(\mathrm{Sd}(\mathcal{I})^{\mathrm{op}})$ is defined if and only if its series Euler characteristic $\chi_{\sum}(\mathcal{I})$ exists, in that case, we obtain
$$\chi_{\sum}(\mathcal{I})=\chi^{(2)}_{\mathrm{ex}}(\mathrm{Sd}(\mathcal{I})^{\mathrm{op}}).$$
\end{mthem}


 We note that the $\mathbb{N}$-filtration $L$ appears on the way to prove our main theorem. The $L^2$-Euler characteristic and the Euler characteristic of $\mathbb{N}$-filtered acyclic categories were independently found, but they are essentially same for categories for categories of the form $\mathrm{Sd}(\mathcal{I})$. When we compute $\chi^{(2)}(\mathrm{Sd}(\mathcal{I})^{\mathrm{op}})$ and $\chi^{(2)}_{\mathrm{ex}}(\mathrm{Sd}(\mathcal{I})^{\mathrm{op}})$, the definition of the $L^2$-Euler characteristic requires us to have a projective resolution of the constant functor $\underline{\mathbb{C}}$ in the functor category $\mathrm{Func}(\mathrm{Sd}(\mathcal{I}), \mathrm{\mathbb{C}\text{-}vect})$. The following is a projective resolution of $\underline{\mathbb{C}}$ we will construct
$$\xymatrix{\dots\ar[r]^{\partial_2}&\displaystyle  \bigoplus_{\mathbf{f_1} \in \overline{N_1}(\mathcal{I})} P_{\mathbf{f_1}}\ar[r]^{\partial_1}&\displaystyle  \bigoplus_{\mathbf{f_0} \in \overline{N_0}(\mathcal{I})} P_{\mathbf{f_0}}\ar[r]^(0.65){\partial_0}&\underline{\mathbb{C}}\ar[r]&0\ar[r]&\dots 
}$$
where $P_{\mathbf{f_n}}$ is a projective object corresponding each $\mathbf{f_n}$ of $\overline{N_n}(\mathcal{I})$ (Note that $\mathbf{f_n}$ is an object in $\mathrm{Sd}(\mathcal{I})$). Thus, this projective resolution gives the $\mathbb{N}$-filtration $L$ on $\mathrm{Sd}(\mathcal{I})$ and conversely $L$ gives the projective resolution.  Furthermore, on the way to compute $$\chi_{\sum}(\mathcal{I}), \chi^{(2)}(\mathrm{Sd}(\mathcal{I})^{\mathrm{op}}), \chi^{(2)}_{\mathrm{ex}}(\mathrm{Sd}(\mathcal{I})^{\mathrm{op}}), \chi_{\mathrm{fil}}(\mathrm{Sd}(\mathcal{I}),L),$$
the power series $$\sum_{n=0}^{\infty} \# \overline{N_n}(\mathcal{I})z^n $$
always appears where $\overline{N_n}(\mathcal{I})$ is the set of non-degenerate chains of morphisms of $\mathcal{I}$. $\chi^{(2)}_{\mathrm{ex}}(\mathrm{Sd}(\mathcal{I})^{\mathrm{op}})$ and $\chi_{\mathrm{fil}}(\mathrm{Sd}(\mathcal{I}),L)$ are just the series Euler characteristic $\chi_{\sum}(\mathcal{I})$ and it can be indicated that the series is very important to consider the Euler characteristic of categories.




This paper is organized as follows.

 In section 2, we give some notations and basic definitions. And we recall the homological algebra over a functor category, which is used in the definition of the $L^2$-Euler characteristic.

In section \ref{L^2}, we prove the four Euler characteristics of categories mentioned above are invariant under the barycentric subdivision for finite acyclic categories. To prove the exactness of the sequence above $$\xymatrix{\dots\ar[r]^{\partial_2}&\displaystyle  \bigoplus_{\mathbf{f_1} \in \overline{N_1}(\mathcal{I})} P_{\mathbf{f_1}}\ar[r]^{\partial_1}&\displaystyle  \bigoplus_{\mathbf{f_0} \in \overline{N_0}(\mathcal{I})} P_{\mathbf{f_0}}\ar[r]^(0.65){\partial_0}&\underline{\mathbb{C}}\ar[r]&0\ar[r]&\dots 
}$$ we introduce the notion of an equivalence $n$-simplex and we prove it forms an acyclic chain complex. 
Finally, we extended the domain of the definition of the $L^2$-Euler characteristic and give a proof of our main theorem.

\section{Preliminaries}

\subsection{Notations}\label{Notations}

\begin{enumerate}

\item Natural numbers mean non-negative integers. 

\item For a natural number $n$, let $[n]=\{0,1,\dots, n\}$ equipped with usual ordering.

\item Let $X$ be a set. Then, $\mathbb{C}[X]$ denotes the free $\mathbb{C}$-vector space generated by $X$.

\item Let $X$ be a finite set. Then, we denote the number of elements of $X$ by $\# X$. 

\item Let $\map{\varphi}{\mathcal{J}}{\mathcal{I}}$ be a functor between small categories and let $i$ be an object of $\mathcal{I}$. Then, the category \textit{$\varphi$-under $i$} is denoted by $(\varphi \downarrow i)$ and the category \textit{$\varphi$-over $i$} is denoted by $(i \downarrow \varphi )$.

\item A \textit{discrete category} $X$ is a category consists of objects and identity morphisms. In particular, if a discrete category has exactly one object, it is called \textit{one-point category}, denoted by $*$.

\item Suppose $\mathcal{J}$ is a small category and $\mathcal{C}$ is a category. The \textit{functor category} $\mathrm{Func}(\mathcal{J}, \mathcal{C})$ consists of functors from $\mathcal{J}$ to $\mathcal{C}$ as its objects and natural transformations between them as its morphisms. Sometimes we simply write it $\mathcal{C}^\mathcal{J}$. 


\end{enumerate}

\subsection{Basic definitions}

In this subsection, we recall basic definitions.

\begin{defi}\label{def}
A small category $\mathcal{A}$ is an \textit{acyclic category} if all the endomorphisms are only identity morphisms and if there exists an arrow $\map{f}{X}{Y}$ such that $X\not = Y$, then there does not exist an arrow $\map{g}{Y}{X}$. Define an order on the set $\mathrm{Ob}(\mathcal{A})$ of objects of $\mathcal{A}$ by $x\le y$ if there exists a morphism $x\rightarrow y$ Then, $\mathrm{Ob}(\mathcal{A})$ is a poset.
\end{defi}

\begin{defi}
Let $\mathcal{J}$ be a small category. The \textit{nerve} $N_*(\mathcal{J})$ of $\mathcal{J}$ is the simplicial set whose set of $n$-simplices $N_n(\mathcal{J})$ is defined as follows \cite{Qui73}:
\begin{eqnarray*}
N_n( \mathcal{J})&=&\{ (f_1,f_2,\dots, f_n)\mid\text{each $f_i$ and $f_{i+1}$ are composable} \}\\
&=&\mathrm{Ob}(\mathrm{Func}([n,\mathcal{J}]))
\end{eqnarray*}

 The \textit{non-degenerate nerve} of $\mathcal{J}$, denoted by $\nn{*}{J}$, is the $\mathbb{N}$-graded subset of $N_*(\mathcal{J})$ 
and each $\nn{n}{J}$ is defined by the following: 
$$\nn{n}{J}=\{ (f_1,f_2,\dots, f_n) \in N_n( \mathcal{J}) \mid\text{none of $f_i$ is the identity morphism} \}$$
where $\nn{0}{J}$ is defined by $\nn{0}{J}=N_0( \mathcal{J})$. 

For any objects $x$ and $y$ of $\mathcal{J}$, define
$$N_n(\mathcal{J})^x_y=\{ \xymatrix{(x_0\ar[r]^{f_1}&x_1\ar[r]^{f_2}&\dots\ar[r]^{f_n}&x_n)} \in N_n(\mathcal{J}) \mid x_0=x, x_n=y \}$$
and $$N_n(\mathcal{J})_y=\{ \xymatrix{(x_0\ar[r]^{f_1}&x_1\ar[r]^{f_2}&\dots\ar[r]^{f_n}&x_n)} \in N_n(\mathcal{J}) \mid  x_n=y \}.$$
$\overline{N_n}(\mathcal{J})^x_y$ and $\overline{N_n}(\mathcal{J})_y$ are also defined in the same way.

\end{defi}

\begin{defi}
Let $\mathcal{J}$ be a small category. Then, \textit{the barycentric subdivision} $\mathrm{Sd}(\mathcal{J})$ of $\mathcal{J}$ is the small category whose objects are the non-degenerate chains of morphisms in $\mathcal{J}$ and the set of morphisms between $X$ and $Y$ is the quotient set of order-preserving maps $\map{f}{[q_X]}{[q_Y]}$ satisfying $Y\circ f=X$ under the relation defined below. Here, $X$ and $Y$ are regarded as functors from posets $[q_X]$ and $[q_Y]$ to $\mathcal{J}$, respectively. So the condition $Y\circ f=X$ implies the commutativity of the diagram 
$$\xymatrix{
&\mathcal{J}&\\
[q_X]\ar[ru]^X\ar[rr]^f&&[q_Y]\ar[lu]_Y
}$$
in the category of small categories.

The equivalence relation is generated by the following relation: Given order-preserving maps $\map{f,g}{[q_X]}{[q_Y]}$ satisfying $Y\circ f=X, Y\circ g=X$, respectively, define $f\sim g$ if for any $0\le i \le q_X$, $Y(\min\{f(i), g(i)\}\rightarrow \max\{f(i), g(i)\})$ is an identity morphism. Here, $$\min\{f(i), g(i)\}\rightarrow \max\{f(i), g(i)\}$$ is a morphism in $[q_Y]$. The composition in Sd$(\mathcal{J})$ is defined by the composition of order-preserving maps.
\end{defi}

We summarize important properties we will often use. For proofs see \cite{Nog11}.
\begin{enumerate}
\item For a small category $\mathcal{J}$, $\mathrm{Sd}(\mathcal{J})$ is an acyclic category.
\item For a morphism $\map{f}{X}{Y}$ in $\mathrm{Sd}(\mathcal{J})$, $\map{f}{[q_X]}{[q_Y]}$ is an order-preserving injection.
\end{enumerate}

\subsection{Homological algebra over a functor category}\label{Homological}

In this subsection, let us recall the definition and basic properties of the Kan extensions. See \cite{KS06},\cite{ML98}, for more details.

Suppose $\map{\varphi}{\mathcal{J}}{\mathcal{I}}$ is a functor between small categories and $\mathscr{C}$ is a category. Then, $\varphi$ induces a functor $\varphi^*$ by precomposition
$$\xymatrix{\mathrm{Func}(\mathcal{I}, \mathscr{C}) \ar[r]^{\varphi^*}&\mathrm{Func}(\mathcal{J}, \mathscr{C}) \ar@/_1pc/[l]_{\varphi^\ddagger} \ar@/^1pc/[l]^{\varphi^\dagger} .}$$
If $\mathscr{C}$ is closed under all small limits and colimits, $\varphi^*$ has a left and a right adjoint $\varphi^\dagger$ and $ \varphi^\ddagger$, respectively. These functors can be described as follows. For any $\map{\beta}{\mathcal{J}}{\mathscr{C}}$

$$\map{\varphi^\dagger (\beta)}{\mathcal{I}}{\mathscr{C}}, \varphi^\dagger (\beta)(i)=\colim \beta\circ P_i$$
$$\map{\varphi^\ddagger (\beta)}{\mathcal{I}}{\mathscr{C}}, \varphi^\ddagger (\beta)(i)=\lim\beta \circ Q^i$$
where $\map{P_i}{(\varphi\downarrow i)}{\mathcal{J}}$ and $\map{Q^i}{(i \downarrow \varphi)}{\mathcal{J}}$ are the projections.

For a morphism $\map{f}{i}{i'}$ in $\mathcal{I}$, $\varphi^\dagger (\beta)(f)$ and $\varphi^\ddagger (\beta)(f)$ are determined by the universal properties. For morphisms 
$$\xymatrix{\varphi(j)\ar[dr]_{g_1}\ar[rr]^{\varphi(h_1)}&&\varphi(j') \ar[dl]^{g_2}\\ &i&},  \xymatrix{&i\ar[dr]^{g_3} \ar[dl]_{g_4}&\\ \varphi(j)\ar[rr]^{\varphi(h_2)}&&\varphi(j') }$$ 
in $(\varphi \downarrow i)$ and $(i\downarrow \varphi)$, respectively, we obtain the following diagrams

$$\xymatrix{\beta(j)=\beta\circ P_i \left( \map{g_1}{\varphi(j)}{i} \right)\ar[d]^{\beta(h_1)}\ar@/^1pc/[drr]^(0.6){\lambda'(f\circ g_1)}\ar[rr]^(0.65){\lambda(g_1)}&&\colim \beta\circ P_i \ar@{.>}[d]^{\exists ! \varphi^\dagger(\beta)(f)} \\
\beta(j')=\beta\circ P_i \left( \map{g_2}{\varphi(j')}{i} \right)\ar@/_1pc/[urr]_(0.65){\lambda(g_2)}\ar[rr]_(0.65){\lambda'(f\circ g_2)}&&\colim \beta\circ P_{i'}
}$$

$$\xymatrix{\beta(j)=\beta\circ Q^i \left( \map{g_3}{i}{\varphi(j)} \right)\ar[d]^{\beta(h_2)}&&\lim \beta \circ Q^i \ar[ll]_(0.35){\mu(g_3 \circ f)} \ar@{.>}[d]^{\exists ! \varphi^\ddagger(\beta)(f)}\ar@/^1pc/[dll]^(0.4){\mu( g_4 \circ f)} \\
\beta(j')=\beta\circ Q^i \left( \map{g_4}{i}{\varphi(j')} \right)&&\lim \beta\circ Q^{i'} \ar[ll]^(0.35){\mu'( g_4)} \ar@/_1pc/[ull]_(0.35){\mu'(g_3)}
}$$
where $\lambda ,\lambda'$ are the limiting cone of $\colim \beta\circ P$ and $\colim \beta\circ P'$ respectively and $\mu, \mu'$ are the limiting cone of $\colim \beta\circ Q$ and $\colim \beta\circ Q'$ respectively.

Since $\varphi^\dagger$ and $\varphi^\ddagger$ are a left and a right adjoint of $\varphi$, respectively, we have the following bijections
 $$\Hom_{\mathrm{Func}(\mathcal{I}, \mathscr{C})}(\varphi^\dagger(\beta), \alpha) \cong \Hom_{\mathrm{Func}(\mathcal{J}, \mathscr{C})}(\beta, \varphi_*(\alpha))$$
 $$\Hom_{\mathrm{Func}(\mathcal{J}, \mathscr{C})}(\varphi_*(\alpha), \beta) \cong \Hom_{\mathrm{Func}(\mathcal{I}, \mathscr{C})}(\alpha, \varphi^\ddagger(\beta)).$$

Recall that for an Abelian category $\mathscr{A}$, the functor category $\mathrm{Func}(\mathcal{I}, \mathscr{A})$ is an Abelian category. The following fact is well-known.

\begin{lemma}\label{projective}
Suppose $\map{\varphi}{\mathcal{J}}{\mathcal{I}}$ is a functor between small categories and $\mathscr{A}$ is an Abelian category closed under all small colimits. If $P$ is a projective object in $\mathrm{Func}(\mathcal{J},\mathscr{A})$, then $\varphi^\dagger (P)$ is projective in $\mathrm{Func}(\mathcal{I},\mathscr{A})$.
\end{lemma}



Let $\mathcal{J}$ be a small category and let $x$ be an object of $\mathcal{J}$. Define $\map{i_x}{*}{\mathcal{J}}$ to be the inclusion functor into $x$. Then we have
$$\xymatrix{\mathrm{Func}(\mathcal{J}, \text{$\mathbb{C}$-vect}) \ar[r]^{i_x^*}&\text{$\mathbb{C}$-vect} \ar@/_1pc/[l]_{i_x^\ddagger} \ar@/^1pc/[l]^{i_x^\dagger} }$$
where $\text{$\mathbb{C}$-vect}$ is the category of $\mathbb{C}$-vector spaces. The comma category $(i_x\downarrow j)$ can be determined easily.

\begin{lemma}\label{discrete}
 Suppose $\mathcal{J}$ is a small category and $x$ is an object of $\mathcal{J}$. For the inclusion functor $\map{i_x}{*}{\mathcal{J}}$ into $x$, $(i_x\downarrow j)$ is the discrete category $\Hom_{\mathcal{J}}(x,j)$ for any object $j$ of $\mathcal{J}$.
\end{lemma}

\begin{prop}
Let $\mathcal{J}$ be a small category. Then, for the functor $$\lmap{i_x^\dagger(\mathbb{C})}{\mathcal{J}}{\text{$\mathbb{C}$-$\mathrm{vect}$}},$$
we have $$i_x^\dagger(\mathbb{C})(j)=\mathbb{C}[\Hom_{\mathcal{J}}(x,j)]$$
and $$\lmap{i_x^\dagger(\mathbb{C})(f)=f^*}{\mathbb{C}[\Hom_{\mathcal{J}}(x,j)]}{\mathbb{C}[\Hom_{\mathcal{J}}(x,j')]}$$
for any object $j$ of $\mathcal{J}$ and for any morphism $\map{f}{j}{j'}$ of $\mathcal{J}$.
\begin{proof}
By Lemma \ref{discrete}, $(i_x\downarrow j)$ is the discrete category $\Hom_{\mathcal{J}}(x,j)$. Hence, we obtain
\begin{eqnarray*}
i_x^\dagger(\mathbb{C})(j)&=& \colim_{(i_x\downarrow j)} \mathbb{C} \\
&=&\mathbb{C}[\Hom_{\mathcal{J}}(x,j)].
\end{eqnarray*}
The universal property of the colimit implies $i_x^\dagger(\mathbb{C})(f)=f^*$.
\end{proof}
\end{prop}

\begin{coro}
Let $\mathcal{J}$ be a small category and $x$ be an object of $\mathcal{J}$. Then, $i_x^\dagger(\mathbb{C})$ is projective in $\mathrm{Func}(\mathcal{J}, \mathrm{\mathbb{C}\text{-}vect})$.
\begin{proof}
This is a special case of Lemma \ref{projective}.
\end{proof}
\end{coro}

\begin{conv}\label{P_x}
For any object $x$ of a small category $\mathcal{J}$, we simply denote $i_x^\dagger(\mathbb{C})$ by $P_x$.
\end{conv}


\section{The invariance of the Euler characteristics under the barycentric subdivision}\label{L^2}

In this section, we prove the four Euler characteristics of categories, Leinster's Euler characteristic \cite{Leia}, the series Euler characteristic \cite{Leib}, the $L^2$-Euler characteristic \cite{FLS} and the Euler characteristic of $\mathbb{N}$-filtered acyclic categories \cite{Nog11} are invariant under the barycentric subdivision for finite acyclic categories. Finally, we introduce an extension of the $L^2$-Euler characteristic and prove our main theorem.

\subsection{Leinster's Euler characteristics of categories}

 Let us recall the definition of Leinster's Euler characteristic \cite{Leia}. Suppose $\mathcal{I}$ is a finite category and the set of objects Ob$(\mathcal{I})$ is labeled by natural numbers as follows.
$$\text{Ob}(\mathcal{I})=\{ x_1, x_2, \dots, x_n\}$$
Let $Z_{\mathcal{I}}$ be the $n\times n$-matrix whose $(i,j)$-entry is the number of morphisms of $\mathcal{I}$ from $x_i$ to $x_j$.

\begin{defi} Let $\mathbf{w,c}$ be row vectors in $\mathbb{Q}^n$. Then, we say $\mathbf{w}$ is a \textit{weighting} on $\mathcal{I}$ if 
$$Z_{\mathcal{I}} {}^{\text{t}} \mathbf{w}=Z_{\mathcal{I}} \begin{pmatrix} w_1\\ w_2 \\ \vdots \\ w_n \end{pmatrix} =\begin{pmatrix} 1 \\ 1 \\ \vdots \\ 1 \end{pmatrix}.$$
We say $\mathbf{c}$ is a \textit{coweighting} on $\mathcal{I}$ if 
$$\mathbf{c}Z_{\mathcal{I}}=(c_1,c_2, \dots, c_n)Z_{\mathcal{I}}=(1,\dots, 1).$$
\end{defi}

\begin{defi}
Define the \textit{Euler characteristic} $\chi_L (\mathcal{I})$ of $\mathcal{I}$ by 
$$\chi_L(\mathcal{I})=\sum^n_{i=1} w_i$$
if $\mathcal{I}$ has both a weighting $\mathbf{w}$ and a coweighting $\mathbf{c}$.

\end{defi}

\begin{defi} We say $\mathcal{I}$ has a \textit{M\"obius inversion} if $Z_{\mathcal{I}}$ has an inverse matrix. Then, the M\"obius inversion $\mu$  is a map
$$\lmap{\mu}{\text{Ob}(\mathcal{I}) \times \text{Ob}(\mathcal{I})}{\mathbb{Q}}$$
defined by $\mu(x_i,x_j)=(i,j)$-entry of $Z_{\mathcal{I}}^{-1}$.
\end{defi}

A finite category $\mathcal{I}$ has a M\"obius inversion if and only if there uniquely exist a weighting and a coweighting on $\mathcal{I}$. Then, we have
$$\sum_{i,j} \mu(x_i,x_j)=\sum^n_{i=1} w_i =\sum^n_{i=1} c_i$$
and $\chi_L(\mathcal{I})=\sum_{i,j} \mu(x_i,x_j)$.

\begin{lemma}\label{iff}
Let $\mathcal{J}$ be a small category. Then, the following are equivalent
\begin{enumerate}
\item $\mathcal{J}$ is finite acyclic.
\item $\mathrm{Sd}(\mathcal{J})$ is a finite category.
\item $\overline{N_k}(\mathcal{J})$ is finite for any $k$ and there exists sufficiently large $M$ such that $\overline{N_n}(\mathcal{J})=\emptyset$ for $n>M$.
\end{enumerate}
\end{lemma}

\begin{prop}\label{Leinster}
Let $\mathcal{I}$ be a finite category. Then, there exists $\chi_L(\mathrm{Sd}(\mathcal{I}))$ if and only if $\mathcal{I}$ is acyclic, in which case, we have $$\chi_L(\mathcal{I})=\chi_L(\mathrm{Sd}(\mathcal{I})).$$
\begin{proof}
Since $\chi_L$ is defined for only finite categories, $\mathrm{Sd}(\mathcal{I})$ must be finite. And since $\mathrm{Sd}(\mathcal{I})$ is acyclic, it has a M\"obius inversion, so there exists $\chi_L(\mathrm{Sd}(\mathcal{I}))$. Hence, $\mathrm{Sd}(\mathcal{I})$ is finite if and only if $\chi_L(\mathrm{Sd}(\mathcal{I}))$ exists. Lemma \ref{iff} proves the first claim. 

Suppose $\mathcal{I}$ is finite acyclic. Then, since $\mathrm{Sd}(\mathcal{I})$ is finite acyclic, we can apply Corollary 1.5 of \cite{Leia} and we obtain a M\"obius inversion $\mu$. We have
\begin{eqnarray}
\chi_L(\mathrm{Sd}(\mathcal{I}))&=& \sum_{\mathbf{f,g}\in \mathrm{Ob}(\mathrm{Sd}(\mathcal{I}))} \mu(\mathbf{f}, \mathbf{g}) \\
&=&\sum_{\mathbf{g}\in \mathrm{Ob}(\mathrm{Sd}(\mathcal{I}))} \left(\sum_{\mathbf{f}\in \mathrm{Ob}(\mathrm{Sd}(\mathcal{I}))} \mu(\mathbf{f}, \mathbf{g}) \right) \nonumber \\
&=&\sum_{\mathbf{g}\in \coprod^M_{n=0}\overline{N_n}(\mathcal{I})} \left(\sum_{\mathbf{f}\in \coprod^{L(\mathbf{g})}_{n=0}\overline{N_n}(\mathcal{I})} \mu(\mathbf{f}, \mathbf{g}) \right) \nonumber  \\
&=&\sum_{\mathbf{g}\in \coprod^M_{n=0}\overline{N_n}(\mathcal{I})} \left( \sum^{L(\mathbf{g})}_{n=0} (-1)^n \# \overline{N_n}(\mathrm{Sd}(\mathcal{I}))_{\mathbf{g}}\right) \label{Noguchi}
\end{eqnarray}
Theorem 4.7 of \cite{Nog11} implies 
$$\sum^{L(\mathbf{g})}_{n=0} (-1)^n \# \overline{N_n}(\mathrm{Sd}(\mathcal{I}))_{\mathbf{g}}=(-1)^{L(\mathbf{g})}.$$
Thus, the equation $(\ref{Noguchi})$ is 
\begin{eqnarray*}
\sum_{\mathbf{g}\in \coprod^M_{n=0}\overline{N_n}(\mathcal{I})} (-1)^{L(\mathbf{g})}&=&\sum^M_{n=0} (-1)^n \# \overline{N_n}(\mathcal{I}) \\
&=&\chi_L(\mathcal{I}).
\end{eqnarray*}
\end{proof}
\end{prop}


\subsection{The series Euler characteristic}

 We recall the \textit{series Euler characteristic} \cite{Leib}.

We have the following commutative diagram of rings.

$$\xymatrix{
\mathbb{Z}[t] \ar@{^{(}->}[d] \ar@{^{(}->}[r]&\mathbb{Z}[[t]] \ar@{^{(}->}[d]&\\
\mathbb{Q}(t)\ar[r]\ar@{^{(}->}[r]&\mathbb{Q}((t))
}$$
Here, $\mathbb{Z}[t]$ is the polynomial ring with the coefficients in $\mathbb{Z}$ and $\mathbb{Z}[[t]]$ is the ring of formal power series over $\mathbb{Z}$. $\mathbb{Q}(t)$ and $\mathbb{Q}((t))$ are the quotient fields of them respectively. 

\begin{defi}

Let $f(t)$ be a formal power series over $\mathbb{Z}$. If there exists a rational function $g(t)/h(t)$ in $\mathbb{Q}(t)$ such that $f(t)=g(t)/h(t)$ in $\mathbb{Q}((t))$, then define 
$$ f|_{t=-1}= \frac{g(-1)}{h(-1)} \in \mathbb{Q}$$
if $h(-1) \not = 0.$
\end{defi}

\begin{defi} Let $\mathcal{I}$ be a finite category. Define the \textit{series Euler characteristic} $\chi_{\sum}(\mathcal{I})$ of $\mathcal{I}$ by
$$\chi_{\sum}(\mathcal{I})=f_{\mathcal{I}}(t)|_{t=-1}$$
where $f_{\mathcal{I}}(t)$ is the formal power series defined by
$$f_{\mathcal{I}}(t)=\sum_{n=0}^{\infty}\# \overline{N_n}(\mathcal{I}) t^n .$$
\end{defi}


\begin{prop}
For a finite category $\mathcal{I}$, $\chi_{\sum}(\mathrm{Sd}(\mathcal{I}))$ can be defined if and only if $\mathcal{I}$ is acyclic, in which case, we obtain $$\chi_{\sum}(\mathcal{I})=\chi_{\sum}(\mathrm{Sd}(\mathcal{I})).$$
\begin{proof}
Since $\mathrm{Sd}(\mathcal{I})$ is acyclic, $\mathrm{Sd}(\mathcal{I})$ is finite if and only if there exists $\chi_{\sum}(\mathrm{Sd}(\mathcal{I}))$. Theorem 3.2 of \cite{Leib} implies $\chi_{L}(\mathrm{Sd}(\mathcal{I}))=\chi_{\sum}(\mathrm{Sd}(\mathcal{I}))$ and Theorem \ref{Leinster} completes this proof.
\end{proof}
\end{prop}


\subsection{The Euler characteristic of $\mathbb{N}$-filtered acyclic categories}

We recall the \textit{Euler characteristic of $\mathbb{N}$-filtered acyclic category} \cite{Nog11}.

\begin{defi}\label{n-fil}\
Let $\mathcal{A}$ be an acyclic category. A functor $\map{\mu}{\mathcal{A}}{\mathbb{N}}$ satisfying $\mu (x)<\mu (y)$ for $x<y$ in Ob$(\mathcal{A})$ is called \textit{an $\mathbb{N}$-filtration of $\mathcal{A}$}. A pair $(\mathcal{A}, \mu )$ is called \textit{an $\mathbb{N}$-filtered acyclic category}.
\end{defi}

\begin{exam}\label{length}
Let $\mathcal{J}$ be a small category. Then, Sd$(\mathcal{J})$ is an acyclic category. The length functor $L$ gives a natural $\mathbb{N}$-filtration to Sd$(\mathcal{J})$ where the functor $L$ is defined by $L(\mathbf{f})=n$ for $\mathbf{f}$ of $\overline{N_n}( \mathcal{J})$. Thus, we obtain an $\mathbb{N}$-filtered acyclic category $(\mathrm{Sd}(\mathcal{J}),L)$.
\end{exam}

\begin{defi}
Let $(\mathcal{A}, \mu)$ be an $\mathbb{N}$-filtered acyclic category. Then, define $\chi_{\text{fil}}(\mathcal{A}, \mu)$ as follows. 

We have the pair of the $\Delta$-set and the natural transformation $$(\overline{N_*}(\mathcal{A}), \overline{N_*}(\mu)).$$ A \textit{$\Delta$-set} is a simplicial set which forgets the degeneracy operators and it is sometimes called \textit{semi-simplicial set}. For an acyclic category $\mathcal{A}$, the pair $$(\overline{N_*}(\mathcal{A}), d|_{\overline{N_*}(\mathcal{A})})$$ becomes a $\Delta$-set since all the$\overline{N_*}(\mathcal{A})$ are closed under the face operator.

Let
$$\overline{N_i}(\mathcal{A})_n=\{ \mathbf{f}\in \overline{N_i}(\mathcal{A})\mid \max(\overline{N_i}(\mu)(\mathbf{f}))=n  \}$$
for natural numbers $i,n$. Suppose each $\overline{N_i}(\mathcal{A})_n$ is finite and $\overline{N_i}(\mathcal{A})_n$ is an empty set if $n<i$. Define the formal power series $f_{\chi}(\mathcal{A}, \mu)(t)$ over $\mathbb{Z}$ by

$$f_{\chi}(\mathcal{A}, \mu)(t)=\sum^\infty_{n=0}(-1)^n\left(\sum^n_{i=0}(-1)^i\# \overline{N_i}(\mathcal{A})_n \right)t^n.$$
And define
$$\chi_{\text{fil}}(\mathcal{A}, \mu)=f_{\chi}(\mathcal{A}, \mu)(t)|_{t=-1}$$
if it exists. 
\end{defi}

\begin{lemma}
Let $\mathcal{A}$ be a finite acyclic category. Then $\mathcal{A}$ has an $\mathbb{N}$-filtration.
\begin{proof}
We can give a linear ordering 
$$\mathrm{Ob}(\mathcal{A})=\{x_1,\dots, x_n\}$$
to the set of objects of $\mathcal{A}$ such that if $x_i< x_j$, then $i<j$ where this ordering is defined in Definition \ref{def}. Indeed, take a maximal element $x$ of $\mathrm{Ob}(\mathcal{A})$ and label it as $x_n$. Inductively, we obtain such labeling. And this labeling gives an $\mathbb{N}$-filtration to $\mathcal{A}$.
\end{proof}
\end{lemma}

\begin{prop}
Let $\mathcal{A}$ be a finite acyclic category. Then, we have
$$\chi_{\mathrm{fil}}(\mathcal{A}, \mu)=\chi_{\mathrm{fil}}(\mathrm{Sd}(\mathcal{A}),L)$$
where $\mu$ is any $\mathbb{N}$-filtration of $\mathcal{A}$ and $L$ is the length $\mathbb{N}$-filtration (see Example \ref{length}).
\begin{proof}
We have 
\begin{eqnarray*}
f_{\chi}(\mathcal{A}, \mu)(t)&=&\sum^\infty_{n=0}(-1)^n\left(\sum^n_{i=0}(-1)^i\# \overline{N_i}(\mathcal{A})_n \right)t^n \\
&=&\sum^M_{n=0}(-1)^n\left(\sum^n_{i=0}(-1)^i\# \overline{N_i}(\mathcal{A})_n \right)t^n
\end{eqnarray*}
for sufficiently large $M$. Hence, $f_{\chi}(\mathcal{A}, \mu)(t)$ is a polynomial. Thus, we obtain
\begin{eqnarray*}
\chi_{\mathrm{fil}}(\mathcal{A}, \mu)&=&f_{\chi}(\mathcal{A}, \mu)(t)|_{t=-1} \\
&=&f_{\chi}(\mathcal{A}, \mu)(-1) \\
&=& \sum^M_{n=0}(-1)^n\left(\sum^n_{i=0}(-1)^i\# \overline{N_i}(\mathcal{A})_n \right)(-1)^n \\
&=& \sum^M_{n=0}\left(\sum^n_{i=0}(-1)^i\# \overline{N_i}(\mathcal{A})_n \right) \\
&=& \sum^M_{i=0}(-1)^i\# \overline{N_i}(\mathcal{A})  \\
&=&\chi_{\sum}(\mathcal{A}).
\end{eqnarray*}
Since $\chi_{\sum}(\mathcal{A})$ exists, Theorem 4.9 of \cite{Nog11} implies 
$$\chi_{\sum}(\mathcal{A})=\chi_{\mathrm{fil}}(\mathrm{Sd}(\mathcal{A}),L).$$
Hence, 
$$\chi_{\mathrm{fil}}(\mathcal{A}, \mu)=\chi_{\mathrm{fil}}(\mathrm{Sd}(\mathcal{A}),L)$$
\end{proof}
\end{prop}


\subsection{The $L^2$-Euler characteristic}

In this subsection, we show the invariance of the $L^2$-Euler characteristic under the barycentric subdivision for finite acyclic categories.

First, we recall the $L^2$\textit{-Euler characteristic} \cite{FLS}. Let $k$ be a commutative ring and let $\mathcal{J}$ be a small category. We denote the category of $k$-modules by $k$-Mod. 

\begin{defi}
If $\map{M}{\mathcal{J}^{\text{op}}}{\text{$k$-Mod}}$ and $\map{N}{\mathcal{J}}{\text{$k$-Mod}}$ are functors, then the \textit{tensor product} $M\otimes_{k\mathcal{J}} N$ is the quotient of the $k$-module
$$\bigoplus_{x \in \text{Ob}(\mathcal{J})} M(x)\otimes_k N(x)$$
by the $k$-submodule generated by elements of the form
$$(M(f^{\mathrm{op}})m)\otimes n -m \otimes(N(f)n)$$
where $\map{f}{x}{y}$ is a morphism in $\mathcal{J}$, $m$ of $M(y)$, and $n$ of $N(x)$. 
\end{defi}

For a discrete group $G$, we denote the \textit{group von Neumann algebra} by $\mathcal{N}(G)$. It is a von Neumann algebra and when $G$ is a finite group $\mathcal{N}(G)$ is just the group ring $\mathbb{C}[G]$. We briefly recall its dimension theory, see \cite{FLS},\cite{Luca},\cite{Lucb} for more details. The \textit{von Neumann dimension} $\dim_{\mathcal{N}(G)}$ is a map which assigns real numbers to left $\mathcal{N}(G)$-modules $$\lmap{\dim_{\mathcal{N}(G)}}{\mathcal{N}(G)\text{-}\mathrm{Mod}}{[0,+\infty]}$$
Here, we ignore the functional analytic aspects of $\mathcal{N}(G)$, so we regard it purely algebraically. An \textit{$\mathcal{N}(G)$-chain complex} is a chain complex of $\mathcal{N}(G)$-modules and its homology is also the usual homology. We often use the fact that when $G$ is a finite group, $\dim_{\mathcal{N}(G)}=\frac{1}{\# G}\dim_{\mathbb{C}}$. For an object $x$ of $\mathcal{J}$, the froup von Neumann algebra $\mathcal{N}(\mathrm{Aut}(x))$ is simply denoted by $\mathcal{N}(x)$.

\begin{defi}
Let $C_*$ be an $\mathcal{N}(G)$-chain complex. The \textit{$p$-th $L^2$-Betti number of $C_*$} is the von Neumann dimension of the $\mathcal{N}(G)$-module given by its $p$-th homology, namely
$$b^{(2)}_p(C_*)= \dim_{\mathcal{N}(G)}(H_p(C_*)) \in [0,\infty].$$
\end{defi}

\begin{defi}
Let $C_*$ be an $\mathcal{N}(G)$-chain complex. Define
$$h^{(2)}(C_*)=\sum_{0\le p}b_p^{(2)}(C_*) \in [0,\infty].$$
If $h^{(2)}(C_*)<\infty$, the \textit{$L^2$-Euler characteristic of $C_*$} is defined by
$$\chi^{(2)}(C_*)=\sum_{0\le p}(-1)^p b_p^{(2)}(C_*) \in \mathbb{R}.$$
\end{defi}
\begin{defi}
Let $\mathcal{J}$ be a small category and let $x$ be an object of $\mathcal{J}$. Define \textit{the splitting functor} at $x$ $$\lmap{S_x}{\mathrm{Func}(\mathcal{J}^{\mathrm{op}}, \mathrm{\mathbb{C}\text{-}vect})}{\mathrm{Func}(\mathrm{Aut}(x)^{\mathrm{op}}, \mathrm{\mathbb{C}\text{-}vect})}$$ as follows. For a functor $\map{F}{\mathcal{J}^{\mathrm{op}}}{\mathrm{\mathbb{C}\text{-}vect}}$, 
$$\lmap{S_x F}{\mathrm{Aut}(x)^{\mathrm{op}}}{\mathrm{\mathbb{C}\text{-}vect}}$$
is defined by 
$$S_x F(*)=\mathrm{Coker}\left( \bigoplus_{\map{u}{x}{y} \ \mathrm{ in }\ \mathcal{J}, \not \exists u^{-1}} F(u^{\mathrm{op}}): \bigoplus_{\map{u}{x}{y} \ \mathrm{ in }\ \mathcal{J}, \not \exists u^{-1}} F(y)\longrightarrow F(x) \right)$$
where this direct sum runs over all the morphisms $\map{u}{x}{y}$ in $\mathcal{J}$ which are not invertible.
For $g^{\mathrm{op}}$ of $\mathrm{Aut}(x)^{\mathrm{op}}$, 
$$S_x(g^{\mathrm{op}}):S_xF(*)\longrightarrow S_x F(*)$$ is defined by $S_x(g^{\mathrm{op}})[m]=[F(g^{\mathrm{op}})(m)]$ for any $[m]$ of $S_xF(*)$.

For a natural transformation $\alpha:F\Rightarrow G$, $S_x \alpha$ is defined by the universal property of the cokernels.
$$\xymatrix{\bigoplus F(y)\ar[d]_{\bigoplus \alpha (y)}\ar[r]^{\bigoplus F(u)}&F(x)\ar[d]^{\alpha(x)}\ar[r]&\mathrm{Coker}=S_x F \ar@{.>}[d]^{\exists ! S_x \alpha}\\
\bigoplus G(y)\ar[r]^{\bigoplus G(u)}&G(x)\ar[r]&\mathrm{Coker}=S_x G
}$$
\end{defi}

\begin{defi}
We call $\mathcal{J}$ \textit{of type} $(L^2)$ if for some projective resolution $P_*$ in $\mathrm{Func}(\mathcal{J}^{\mathrm{op}}, \mathrm{\mathbb{C}\text{-}vect})$ of the constant functor $\underline{\mathbb{C}}$ we have
$$h^{(2)}(\mathcal{J})= \sum_{[x]\in \text{iso($\mathcal{J})$}} h^{(2)} (S_x P_* \otimes_{\mathbb{C}[x]} \mathcal{N}(x))<\infty .$$
\end{defi}

\begin{defi}
Suppose that $\mathcal{J}$ is of type $(L^2)$. Define the \textit{$L^2$-Euler characteristic} of $\mathcal{J}$ to be the real number
$$\chi^{(2)}(\mathcal{J})= \sum_{[x]\in \text{iso($\mathcal{J})$}} \chi^{(2)} (S_x P_* \otimes_{\mathbb{C}[x]} \mathcal{N}(x))\in \mathbb{R},$$
where $P_*$ is a projective resolution of the constant functor $\underline{\mathbb{C}}$ in $\mathrm{Func}(\mathcal{J}^{\mathrm{op}}, \mathrm{\mathbb{C}\text{-}vect})$.

Notice that this definition makes sense since the condition $(L^2)$ ensures that the sum $\sum_{[x]\in \text{iso($\mathcal{J})$}} \chi^{(2)} (S_x P_* \otimes_{\mathbb{C}[x]} \mathcal{N}(x))$ is absolutely convergent.
\end{defi}

The following is our main theorem of this section and the proof is given later.

\begin{them}\label{moto-main}
For a small category $\mathcal{I}$, $\chi^{(2)}(\mathrm{Sd}(\mathcal{I})^{\mathrm{op}})$ exists if and only if $\mathcal{J}$ is finite acyclic, in which case, we obtain $$\chi^{(2)}(\mathrm{Sd}(\mathcal{I})^{\mathrm{op}})=\chi^{(2)}(\mathcal{I}).$$
\end{them}

To prove this theorem we need Lemma \ref{AS} and Proposition \ref{PR}. In lemma \ref{AS}, we characterize the splitting functor for an acyclic category. And in Proposition \ref{PR}, we construct a projective resolution of $\underline{\mathbb{C}}$ in $\mathrm{Func}(\mathrm{Sd}(\mathcal{J}), \mathbb{C}\text{-}\mathrm{vect})$.

For any object $x$ of a small category $\mathcal{J}$, we simply denote $i_x^\dagger(\mathbb{C})$ by $P_x$ (see Convention \ref{P_x}).

\begin{lemma}\label{AS}
Let $\mathcal{A}$ be an acyclic category and $x$ and $y$ be objects of $\mathcal{A}$. For the functor
$$\lmap{S_x}{\mathrm{Func}(\mathcal{A}^{\mathrm{op}}, \mathrm{\mathbb{C}\text{-}vect})}{ \mathrm{\mathbb{C}\text{-}vect}},$$
we have
$$S_x P_y=\begin{cases} \mathbb{C} &\text{if } x=y \\ 0&\text{if } x\not =y\end{cases}.$$
\begin{proof}
For the functor $\map{P_y}{\mathcal{A}^{\mathrm{op}}}{\mathbb{C}\text{-vect}}$, we have 
\begin{eqnarray*}
P_y(z)&=&\mathbb{C}[\Hom_{\mathcal{A}^{\mathrm{op}}}(y,z)] \\
&=&\mathbb{C}[\Hom_{\mathcal{A}}(z,y)]
\end{eqnarray*}
for an object $z$ of $\mathcal{A}^{\mathrm{op}}$.
We have 
\begin{eqnarray*}
S_x P_y&=& \mathrm{Coker}\left( \bigoplus_{\map{u}{x}{z} \ \mathrm{ in }\ \mathcal{A}, \not \exists u^{-1}} P_y(u^{\mathrm{op}}): \bigoplus_{\map{u}{x}{z} \ \mathrm{ in }\ \mathcal{A}, \not \exists u^{-1}} P_y (z)\longrightarrow P_y(x) \right) \\
&=&\mathrm{Coker}\left( \bigoplus_{\begin{matrix}\map{u}{x}{z} \\ u\not =1\end{matrix}} u^*: \bigoplus_{\begin{matrix}\map{u}{x}{z} \\ u\not =1\end{matrix}} \mathbb{C}[\Hom_{\mathcal{A}}(z,y)] \rightarrow \mathbb{C}[\Hom_{\mathcal{A}}(x,y)] \right) .
\end{eqnarray*}
If $x=y$, then 
\begin{eqnarray*}
S_x P_x&=& \mathrm{Coker}\left( \bigoplus_{\begin{matrix}\map{u}{x}{z} \\ u\not =1\end{matrix}} u^*: \bigoplus_{\begin{matrix}\map{u}{x}{z} \\ u\not =1\end{matrix}} \mathbb{C}[\Hom_{\mathcal{A}}(z,x)] \rightarrow \mathbb{C} \right) .
\end{eqnarray*}
Here, all of the running $\map{u}{x}{z}$ are not $1_x$, so $x\not =z$. Since $\mathcal{A}$ is acyclic,  $\Hom_\mathcal{A}(z,x)$ are empty-sets if there exists a morphism $\map{u}{x}{z}$. Hence, \begin{eqnarray*}
S_x P_x&=& \mathrm{Coker}\left( 0: 0 \rightarrow \mathbb{C} \right) \\
&=&\mathbb{C}.
\end{eqnarray*}
Suppose $x\not = y$. If $\Hom_\mathcal{A}(x,y)=\emptyset$, then we obtain  
\begin{eqnarray*}
S_x P_y&=& \mathrm{Coker}\left( \bigoplus_{\begin{matrix}\map{u}{x}{z} \\ u\not =1\end{matrix}} u^*: \bigoplus_{\begin{matrix}\map{u}{x}{z} \\ u\not =1\end{matrix}} \mathbb{C}[\Hom_{\mathcal{A}}(z,y)] \rightarrow 0 \right)\\
&=&0.
\end{eqnarray*}
If $\Hom_\mathcal{A}(x,y)\not =\emptyset$, then such $\map{u}{x}{z}$ runs over $\Hom_{\mathcal{A}}(x,y)$ and $$\coprod_{\begin{matrix}\map{u}{x}{z} \\ u\not =1\end{matrix}} \Hom_{\mathcal{A}}(z,y)$$
contains $1_y$, hence we obtain $S_x P_y=0$.
\end{proof}
\end{lemma}

Next, we begin with the first part. Let $\mathcal{J}$ be a small category. We construct a projective resolution of the constant functor $\underline{\mathbb{C}}$ in $\mathrm{Func}(\mathrm{Sd}(\mathcal{J}), \mathbb{C}\text{-}\mathrm{vect})$. Let $P(\mathrm{Sd}(\mathcal{J}))_*$ be the sequence
$$\xymatrix{\dots\ar[r]^{\partial_2}&\displaystyle  \bigoplus_{\mathbf{f_1} \in \overline{N_1}(\mathcal{J})} P_{\mathbf{f_1}}\ar[r]^{\partial_1}&\displaystyle  \bigoplus_{\mathbf{f_0} \in \overline{N_0}(\mathcal{J})} P_{\mathbf{f_0}}\ar[r]^(0.65){\partial_0}&\underline{\mathbb{C}}\ar[r]&0\ar[r]&\dots 
}$$
where each $\partial_k$ is defined as follows. For $\mathbf{g}$ of $\overline{N}_k(\mathcal{J})$, we have $$\bigoplus_{\mathbf{f_k} \in \overline{N_k}(\mathcal{J})} P_{\mathbf{f_k}}(\mathbf{g})=\bigoplus_{\mathbf{f_k} \in \overline{N_k}(\mathcal{J})} \mathbb{C}[\Hom_{\mathrm{Sd}(\mathcal{J})} (\mathbf{f_k}, \mathbf{g})].$$
The map $$\lmap{\partial_k(\mathbf{g})}{\bigoplus_{\mathbf{f_k} \in \overline{N_k}(\mathcal{J})} \mathbb{C}[\Hom_{\mathrm{Sd}(\mathcal{J})} (\mathbf{f_{k-1}}, \mathbf{g})]}{\bigoplus_{\mathbf{f_{k-1}} \in \overline{N_{k-1}}(\mathcal{J})} \mathbb{C}[\Hom_{\mathrm{Sd}(\mathcal{J})} (\mathbf{f_{k-1}}, \mathbf{g})]}$$
is defined by
$$\partial_k(\mathbf{g})(\varphi)=\sum_{j\in F(\mathbf{f_k})}(-1)^j \varphi\circ d^j$$ for any $\varphi$ of $\Hom_{\mathrm{Sd}(\mathcal{J})} (\mathbf{f_{k-1}}, \mathbf{g})$ where $$F(\mathbf{f_k})=\{ j\in [k]\mid d_j(\mathbf{f_k}) \in \overline{N_{k-1}} (\mathcal{J})\}.$$ For a morphism $\map{f}{\mathbf{g}}{\mathbf{g'}}$ in $\mathrm{Sd}(\mathcal{J})$, the following diagrams are commutative
$$\xymatrix{\displaystyle \bigoplus_{\mathbf{f_k} \in \overline{N_k}(\mathcal{J})} \mathbb{C}[\Hom_{\mathrm{Sd}(\mathcal{J})} (\mathbf{f_k}, \mathbf{g})]\ar[d]^{f_*}\ar[rr]^{\partial_k(\mathbf{g})}&& \displaystyle \bigoplus_{\mathbf{f_{k-1}} \in \overline{N_{k-1}}(\mathcal{J})} \mathbb{C}[\Hom_{\mathrm{Sd}(\mathcal{J})} (\mathbf{f_{k-1}}, \mathbf{g})] \ar[d]^{f_*}\\
\displaystyle \bigoplus_{\mathbf{f_k} \in \overline{N_k}(\mathcal{J})} \mathbb{C}[\Hom_{\mathrm{Sd}(\mathcal{J})} (\mathbf{f_k}, \mathbf{g'})]\ar[rr]^{\partial_k(\mathbf{g'})}&& \displaystyle \bigoplus_{\mathbf{f_{k-1}} \in \overline{N_{k-1}}(\mathcal{J})} \mathbb{C}[\Hom_{\mathrm{Sd}(\mathcal{J})} (\mathbf{f_{k-1}}, \mathbf{g'})].
}$$
$$\xymatrix{\varphi\ar@{|-{>}}[rr]^{\partial_k(\mathbf{g})}\ar@{|-{>}}[d]^{f_*}&&\sum_{j\in F(\mathbf{f_k})} \varphi\circ d^j \ar@{|-{>}}[d]^{f_*}\\
f\circ \varphi \ar@{|-{>}}[rr]^{\partial_k(\mathbf{g'})}&&\sum_{j\in F(\mathbf{f_k})} f\circ \varphi\circ d^j
}$$
for $\map{\varphi}{\mathbf{f_k}}{\mathbf{g}}$ of $\Hom_{\mathrm{Sd}(\mathcal{J})}(\mathbf{f_k}, \mathbf{g})$.
Therefore, $\partial_k$ is a natural transformation. At $k=0$, $\partial_0$ is the augmentation, that is, for $\mathbf{g}$ of $\overline{N_k}(\mathcal{J})$, 
$$\lmap{\partial_0(\mathbf{g})}{\bigoplus_{\mathbf{f_{0}} \in \overline{N_{0}}(\mathcal{J})} \mathbb{C}[\Hom_{\mathrm{Sd}(\mathcal{J})} (\mathbf{f_{0}}, \mathbf{g})]}{\mathbb{C}}$$
$\partial_0(\mathbf{g})(\varphi)=1$ for any $\varphi$ of $\Hom_{\mathrm{Sd}(\mathcal{J})} (\mathbf{f_{0}}, \mathbf{g})$.

To prove $P(\mathrm{Sd}(\mathcal{J}))_*$ is exact we introduce the notion of \textit{equivalence $n$-simplex}. It is a generalization of a combinatorial $n$-simplex and it is obtained by exclusion and identification of some faces of an $n$-simplex. We prove that an equivalence $n$-simplex generates an acyclic chain complex and this fact implies $P(\mathrm{Sd}(\mathcal{J}))_*$ is exact.

\begin{defi}
Let $n$ be a natural number. Suppose $\sim$ is an equivalence relation on $[n]$ with the property that if $i\sim j$, then $i+1\not = j$ and $i\not = j+1$. Let 
$$A_k^{(n)}=\{(i_0,i_1, \dots , i_k) \in [n]^{k+1}\mid i_0 < \dots <i_k \}$$
and
$$B_k^{(n)}=\{(i_0,i_1, \dots , i_k) \in A_k^{(n)}\mid \exists m \text{ s.t. } i_m \sim i_{m+1}\}$$
and 
$$C_k^{(n)}=(A_k^{(n)}-B_k^{(n)})/\approx$$ 
where $ (i_0,i_1, \dots , i_k) \approx (j_0,j_1, \dots , j_k)$ is defined by $i_m \sim j_m$ for any $m$. 
We call the family $\{C_k^{(n)}\}_{k\ge -1}$ an \textit{equivalence $n$-simplex}. For $k=-1$ let $A_{-1}^{(n)}=C_{-1}^{(n)}=*$ and let $B_{-1}^{(n)}=\emptyset$.

Note that since $A_{n}^{(n)}=\{(0,1,\dots,n)\}$ , the property of the equivalence relation implies $B_{n}^{(n)}=\emptyset$. Hence, $C_{n}^{(n)}=*$.
\end{defi}

\begin{exam}
Suppose $n=2$ and $0\sim 2$. Then, we have
\begin{eqnarray*}
C_0^{(2)}&=&\{ [(0)]=[(2)], [(1)]\} \\
C_1^{(2)}&=&\{ [(0,1)],[(1,2)]\} \\
C_0^{(2)}&=&\{ [(0,1,2)]\}.
\end{eqnarray*}
$\{A_k^{(2)}\}$ and $\{C_k^{(2)}\}$ are visualized as follows.

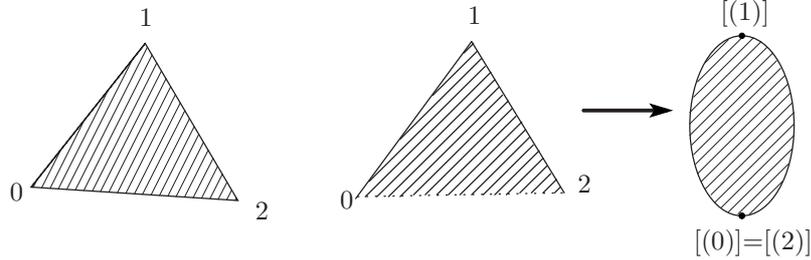
\input{eq.n-simp}

The left hand side is $\{A_k^{(2)}\}$ and the right hand side is $\{C_k^{(2)}\}$.
\end{exam}

The \textit{face operator} $\map{d_j}{A^{(n)}_k}{A^{(n)}_{k-1}}$ is the map to eliminate the $j$-th coordinate, $$d_j(i_0,i_1,\dots, i_k)=(i_0,\dots,i_{j-1},i_{j+1},\dots, i_k).$$ It is partially defined on $C_k^{(n)}$. We give the definition in the following.

\begin{lemma} Let $\{C_k^{(n)}\}_{k\ge -1}$ be an equivalence $n$-simplex. For $[(i_0,\dots,i_k)]$ of $C_k^{(n)}$, define 
$$F([(i_0,\dots,i_k)])=\{ \ell \in [k] \mid d_{\ell}(i_0,\dots,i_k) \not \in B_{k-1}^{(n)}\}.$$
Then, $F([(i_0,\dots,i_k)])$ does not depend on the choice of the representation of $[(i_0,\dots,i_k)]$.
\begin{proof}
Suppose $(i_0,\dots,i_k) \approx (j_0,\dots, j_k)$. For $\ell$ in $F([(i_0,\dots,i_k)])$, we have $i_{l-1}\not \sim i_{l+1}$. Then we also have $j_{l-1}\not \sim j_{l+1}$, since $j_{l-1}\sim i_{l-1} \not \sim i_{l+1}\sim j_{l+1}.$ Hence, $F([(i_0,\dots,i_k)])$ contains $l$ if and only if $F([(j_0,\dots,j_k)])$ contains $l$,
$$F([(i_0,\dots,i_k)])=F([(j_0,\dots,j_k)]).$$
\end{proof}
\end{lemma}

\begin{defi}
Let $\{C_k^{(n)}\}$ be an equivalence $n$-simplex. For $[(i_0,i_1,\dots, i_k)]$ of $C_k^{(n)}$ and $\ell$ of $F([(i_0,\dots,i_k)])$, define $$d_l([(i_0,i_1,\dots, i_k)])=[d_{\ell}(i_0,\dots,i_k)].$$ 

If $(i_0,\dots, i_k) \approx (j_0,\dots, j_k)$, then $i_m\sim j_m$ for any $m$. So $$(i_0,\dots,i_{\ell-1},i_{\ell+1},\dots, i_k) \approx (j_0,\dots,j_{\ell-1},j_{\ell+1},\dots, j_k).$$ Hence, this map is well-defined.
\end{defi}

\begin{defi}
Let $\{C_k^{(n)}\}_{k\ge -1}$ be an equivalence $n$-simplex. For $0< k$ define $\map{D_k}{\mathbb{C}[C_k^{(n)}]}{\mathbb{C}[C_{k-1}^{(n)}]}$ by
$$D_k([(i_0,\dots,i_k)])=\sum_{j\in F([(i_0,\dots,i_k)])} (-1)^j d_j([(i_0,\dots,i_k)])$$
for any $[(i_0,\dots,i_k)]$ of $C_k^{(n)}$. For $k=0$ define $\map{D_0}{\mathbb{C}[C_0^{(n)}]}{\mathbb{C}}$ to be the augmentation, that is, $$D_0(\sum_{x_i \in C_0^{(n)}} \alpha_i) x_i =  \sum_{x_i \in C_0^{(n)}} \alpha_i .$$
\end{defi}

\begin{prop}
Let $\{C_k^{(n)}\}_{k\ge -1}$ be an equivalence $n$-simplex. Then, $D_{k-1}\circ D_k=0$. Hence,
$$\xymatrix{\dots \ar[r]&0\ar[r]&\mathbb{C} \ar[r]^{D_n}&\mathbb{C}[C_{n-1}^{(n)}]\ar[r]^{D_{n-1}}&\dots  \\
\dots \ar[r]^{D_2}& \mathbb{C}[C_1^{(n)}] \ar[r]^{D_1}& \mathbb{C}[C_0^{(n)}] \ar[r]^{D_0}&\mathbb{C} \ar[r]&0\ar[r]&\dots}$$
 is a chain complex.
\begin{proof}
We prove this claim by comparing this complex with the familiar chain complex $\{ \mathbb{C}[A_k^{(n)}] ,\partial_k\}_{k\ge -1}$ where $\map{\partial_k}{\mathbb{C}[A_k^{(n)}]}{\mathbb{C}[A_{k-1}^{(n)}]}$ is defined by the alternating sum of the face operators, $$\partial_k=\sum_{j=0}^k (-1)^j d_j.$$ This chain complex is isomorphic to the augmented  chain complex of usual $n$-simplex with coefficients in $\mathbb{C}$.

Define a map $\map{p_k}{\mathbb{C}[A_k^{(n)}]}{\mathbb{C}[C_k^{(n)}]}$ by
$$p_k((i_0,\dots, i_k))=\begin{cases} [(i_0,\dots,i_k)] &\text{if}\ (i_0,\dots, i_k) \not \in B_k^{(n)} \\ 0&\text{if} \ (i_0,\dots, i_k)  \in B_k^{(n)}. \end{cases}$$
In particular, define $p_{-1}=1_{\mathbb{C}}$. We show $\{p_k\}:\{\mathbb{C}[A_k^{(n)}] ,\partial_k\}\rightarrow \{\mathbb{C}[C_k^{(n)}] ,D_k\}$ is a chain map. It suffices to show that the following two types of diagrams are commutative 
$$\xymatrix{ \mathbb{C}[A_{k}^{(n)}]\ar[d]^{p_k}\ar[r]^{\partial_k}&\mathbb{C}[A_{k-1}^{(n)}] \ar[d]^{p_{k-1}}&&\mathbb{C}[A_{0}^{(n)}] \ar[d]^{p_0}\ar[r]^{\partial_0} &\mathbb{C} \ar[d]^{1}\\  
\mathbb{C}[C_{k}^{(n)}]\ar[r]^{D_k}&\mathbb{C}[C_{k-1}^{(n)}]&&\mathbb{C}[C_{0}^{(n)}] \ar[r]^{D_0} &\mathbb{C} 
}$$
where $\partial_0$ is the augmentation and $1\le k\le n$.

Since $B_0^{(n)}=\emptyset$, $p_0$ is a natural projection. So $p_0$ does not vanish any elements of $\mathbb{C}[A_0^{(n)}]$. Hence, the diagram of the right hand side is commutative.  

Next we show the commutativity of the left hand side. Take $(i_0,\dots,i_k)$ of $A_k^{(n)}$. Suppose $B_k^{(n)}$ contains it. We have 
\begin{eqnarray*}
D_k\circ p_k ((i_0,\dots ,i_k))&=&D_k(0) \\
&=&0.
\end{eqnarray*}
Since $B_k^{(n)}$ contains $(i_0,\dots,i_k)$, there exists $\ell$ such that $0\le \ell<k$ and $i_{\ell}\sim i_{\ell+1}$. Here, we have to consider two cases,
\begin{enumerate}
\item the existence of such $\ell$ is unique
\item there is another such $\ell'$.
\end{enumerate}
In the first case, 
\begin{eqnarray*}
p_{k-1}\circ \partial_k ((i_0,\dots ,i_k))&=&p_{k-1} \left( \sum^k_{j=0} (-1)^j d_j(i_0,\dots,i_k) \right).
\end{eqnarray*}
For $0\le j\le \ell-1$ or $\ell+2\le j\le k$, $d_j(i_0,\dots ,i_k)$ contains $i_{\ell}$ and $i_{\ell+1}$ which are next to, hence $p_{k-1}(d_j(i_0\dots, i_k))=0. $ Since $i_{\ell}\sim i_{\ell+1}$, $d_{\ell} (i_0,\dots,i_k)\approx d_{\ell+1} (i_0,\dots,i_k)$. This fact implies
\begin{eqnarray*}
p_{k-1}\circ \partial_k ((i_0,\dots ,i_k))&=&p_{k-1} \left( (-1)^{\ell} d_{\ell}(i_0,\dots,i_k) + (-1)^{\ell+1} d_{\ell+1}(i_0,\dots,i_k)\right) \\
&=&(-1)^{\ell} [d_{\ell}(i_0,\dots,i_k)] + (-1)^{\ell+1} [d_{\ell+1}(i_0,\dots,i_k)] \\
&=&0.
\end{eqnarray*}
In the second case, $p_{k-1}$ vanishes all the terms of $\sum^k_{j=0} (-1)^j d_j(i_0,\dots,i_k)$.

If $B_k^{(n)}$ does not contain $(i_0,\dots,i_k)$, it is easy to see $$p_{k-1}\circ \partial_k ((i_0\dots, i_k))=D_k\circ p_k((i_0,\dots, i_k)).$$

Hence, $\{p_k \}$ is a chain map. Since each $p_k$ is a surjection and $\partial_{k-1}\circ \partial_k=0$, we obtain $D_{k-1}\circ D_k=0$.
\end{proof}
\end{prop}

\begin{prop}\label{DMT}
$$H_m (\{C_k^{(n)} , D_k\}_{k\ge -1})=0$$
for any $m$.
\begin{proof}
Define a contracting homotopy $\map{h_k}{\mathbb{C}[C_k^{(n)}]}{\mathbb{C}[C_{k+1}^{(n)}]}$ by 
$$h_k([(i_0,\dots, i_k)])=\begin{cases}[(0,i_0,\dots, i_k)]&\text{if } 0\not \sim i_0 \\
0&\text{if } 0\sim i_0.
\end{cases}$$
In particular, for $k=-1$ define $\map{h_{-1}}{\mathbb{C}}{\mathbb{C}[C_{0}^{(n)}]}$ by $h_{-1}(*)=[(0)]$. Then we have the following diagram
$$\xymatrix{
\dots\ar[r]&\mathbb{C}[C_{2}^{(n)}]\ar[d]^1\ar[r]^{D_2}\ar[dl]_{h_{2}}&\mathbb{C}[C_{1}^{(n)}]\ar[d]^1\ar[dl]_{h_{1}}\ar[r]^{D_1}&\mathbb{C}[C_{0}^{(n)}]\ar[d]^1\ar[dl]_{h_{0}}\ar[r]^{D_0}&\mathbb{C}\ar[r]\ar[d]^1\ar[dl]_{h_{-1}}&0\ar[r]\ar[d]\ar[dl]&\dots \\
\dots\ar[r]&\mathbb{C}[C_{2}^{(n)}]\ar[r]^{D_2}&\mathbb{C}[C_{1}^{(n)}]\ar[r]^{D_1}&\mathbb{C}[C_{0}^{(n)}]\ar[r]^{D_0}&\mathbb{C}\ar[r]&0\ar[r]&\dots .
}$$ We have 
\begin{eqnarray*}
D_0\circ h_{-1}(1)&=& D_0[(0)]\\
&=&1.
\end{eqnarray*}
For $0\le k\le n$, we show $h_{k-1}\circ D_k+D_{k+1}\circ h_k=1$. Take an element $[(i_0,\dots, i_k)]$ of $C_k^{(n)}$. If $i_0\sim 0$, then
\begin{multline}
(h_{k-1}\circ D_k+D_{k+1}\circ h_k)([(i_0,\dots, i_k)])=h_{k-1}\circ D_k([(i_0,\dots, i_k)])  \\
=h_{k-1}\left( \sum_{j\in F([(i_0,\dots, i_k)])} (-1)^j d_j[(i_0,\dots,i_k)] \right) \label{homotopy}
\end{multline}
Here, for $0<j \in F([(i_0,\dots, i_k)])$, we have
\begin{eqnarray*}
h_{k-1}( (-1)^j d_j [ (i_0, \dots, i_k) ])&=&h_{k-1} ((-1)^j [ (i_0, \dots,i_{j-1},i_{j+1},\dots, i_k) ])\\
&=&0.
\end{eqnarray*}
Since $i_0\not \sim i_1$, we have $i_1\not \sim 0$. Thus, the equation $(\ref{homotopy})$ is 
\begin{eqnarray*}
h_{k-1}( (-1)^0 d_0 [ (i_0, \dots, i_k) ])&=&h_{k-1}( [ (i_1, \dots, i_k) ])\\
&=&[ (0,i_1, \dots, i_k) ] \\
&=&[ (i_0,i_1, \dots, i_k) ].
\end{eqnarray*}
If $i_0\not \sim 0$, then we have
\begin{eqnarray*}
D_{k+1}\circ h_k ([(i_0,\dots, i_k)])&=&D_{k+1}([(0,i_0,\dots, i_k)]) \\
&=&\sum_{j\in F([(0, i_0,\dots, i_k)])} (-1)^j d_j[(0, i_0,\dots,i_k)]
\end{eqnarray*}
and 
\begin{eqnarray*}
h_{k-1}\circ D_k ([(i_0,\dots, i_k)])&=&h_{k-1}\left( \sum_{j\in F([(i_0,\dots, i_k)])} (-1)^j d_j[(i_0,\dots,i_k)] \right).
\end{eqnarray*}
For $j>0$, $D([(i_0,\dots, i_k)])$ contains $j$ if and only if $D([(0, i_0,\dots, i_k)])$ contains $j+1$. Thus, we obtain
\begin{eqnarray*}
(h_{k-1}\circ D_k+D_{k+1}\circ h_k=1)([(i_0,\dots, i_k)])&=& (-1)^0 d_0[(0, i_0,\dots, i_k)]  \\
&=& [(i_0,\dots, i_k)].
\end{eqnarray*}
We conclude $\{C_k^{(n)},D_k\}$ is a split exact sequence.
\end{proof}
\end{prop}

This result is a homological interpretation of Proposition 4.6 of \cite{Nog11} which proved the \textit{reduced Euler characteristic} of an equivalence $n$-simplex $\{C_k^{(n)}\}$ is zero, that is, $$\chi(\{C_k^{(n)}\})=\sum^n_{k=-1} (-1)^k \# C_k^{(n)}=0.$$

\begin{prop}\label{PR}
For a small category $\mathcal{J}$, $P(\mathrm{Sd}(\mathcal{J}))_*$ is a projective resolution of $\underline{\mathbb{C}}$ in $\mathrm{Func}(\mathrm{Sd}(\mathcal{J}), \mathbb{C}\text{-}\mathrm{vect})$.
\begin{proof}
Since each $P_{\mathbf{f}}$ is projective for any object $\mathbf{f}$ of $\mathrm{Sd}(\mathcal{J})$, $\bigoplus_{\mathbf{f_k} \in \overline{N_k}(\mathcal{J})} P_{\mathbf{f_k}}$ is also projective for any $k$. Next we show exactness of $P(\mathrm{Sd}(\mathcal{J}))_*$. Note that $P(\mathrm{Sd}(\mathcal{J}))_*$ is exact if and only if each $P(\mathrm{Sd}(\mathcal{J}))_*(\mathbf{g})$ is exact for any $\mathbf{g}$ of $\overline{N_n}(\mathcal{J})$. Take $\mathbf{g}$ of $\overline{N_n}(\mathcal{J})$ and define an equivalence relation $\sim_{\mathbf{g}}$ on $[n]$ by $i\sim j$ if 
$$\mathbf{g}(\min\{i,j\}\rightarrow \min\{i,j\})=1.$$
Then, $\sim_{\mathbf{g}}$ is an equivalence relation and it satisfies $i\not\sim i+1$ for any $i$. For this equivalence relation, we obtain an equivalence $n$-simplex and its chain complex $\{C_k^{(n)}, D_k\}$. Then, the chain complex is isomorphic to $P(\mathrm{Sd}(\mathcal{J}))_*(\mathbf{g})$. Define two maps
$$\lmap{\varphi_k}{C_k^{(n)}}{\coprod_{\mathbf{g} \in \overline{N_k}(\mathcal{J})} \Hom_{\mathrm{Sd}(\mathcal{J})} (\mathbf{f_k}, \mathbf{g})}$$
$$\lmap{\psi_k}{\coprod_{\mathbf{g} \in \overline{N_k}(\mathcal{J})} \Hom_{\mathrm{Sd}(\mathcal{J})} (\mathbf{f_k}, \mathbf{g})}{C_k^{(n)}}$$
by $$\varphi_k([(i_0,\dots,i_k)]):[k]\longrightarrow[n]$$
$$\varphi_k([(i_0,\dots,i_k)])(j)=i_j$$
and
$$\psi_k(\alpha)=[(\alpha(0),\dots, \alpha(k))]$$
for any $[(i_0,\dots, i_k)]$ of $C_k^{(n)}$ and any $\map{\alpha}{\mathbf{f}_k}{\mathbf{g}}$. In general, a morphism $\map{f}{X}{Y}$ in $\mathrm{Sd}(\mathcal{J})$ satisfies $X=Y\circ f$, so $Y$ and $f$ determine $X$. Thus, the order-preserving injection $\varphi_k([(i_0,\dots,i_k)])$ and $\mathbf{g}$ determine the domain of the map $\map{\varphi_k([(i_0,\dots,i_k)])}{?}{\mathbf{g}}$. Then, $\varphi_k$ and $\psi_k$ are well-defined.
Indeed, if $\map{\alpha_1\sim \alpha_2}{\mathbf{f_k}}{\mathbf{g}}$, then $$\mathbf{g}(\min\{\alpha_1(i),\alpha_2(i)\}\rightarrow \min\{\alpha_1(i),\alpha_2(i)\})=1)$$
for any $i$, that is, $\alpha_1(i)\sim_{\mathbf{g}} \alpha_2(i)$. Hence, 
\begin{eqnarray*}
\psi_k(\alpha_1)&=&[(\alpha_1(0),\dots,\alpha_1(k))] \\
&=&[(\alpha_2(0),\dots,\alpha_2(k))] \\
&=&\psi(\alpha_2).
\end{eqnarray*}
If $[(i_0,\dots, i_k)]=[(j_0,\dots, j_k)]$, then $i_{\ell}\sim_{\mathbf{g}} j_{\ell}$ for any $\ell$. So we have $$\mathbf{g}\left( \min \{ i_{\ell},j_{\ell} \}\rightarrow \max \{ i_{\ell},j_{\ell} \}\right),$$
we have $\varphi_k([(i_0,\dots,i_k)])\sim \varphi_k([(j_0,\dots,j_k)])$. It is clear that $\varphi\circ \psi=1$ and $\psi\circ\varphi=1$. Moreover, $\{\varphi_k\}$ is compatible with the differentials, so $\{\varphi_k\}$ is a chain map. Hence, $P(\mathrm{Sd}(\mathcal{J}))_*(\mathbf{g})$ is isomorphic to $\{C_k^{(n)},D_k\}$. Proposition \ref{DMT} implies $\{C_k^{(n)},D_k\}$ is exact, so is  $P(\mathrm{Sd}(\mathcal{J}))_*(\mathbf{g})$.
\end{proof} 
\end{prop}

Finally, we give a proof of Theorem \ref{moto-main}.

\begin{proof}[Proof of Theorem \ref{moto-main}]
To compute $\chi^{(2)}(\mathrm{Sd}(\mathcal{I})^{\mathrm{op}})$ we work on the category $$\mathrm{Func}((\mathrm{Sd}(\mathcal{J})^{\mathrm{op}})^{\mathrm{op}}, \mathbb{C}\text{-}\mathrm{vect})=\mathrm{Func}(\mathrm{Sd}(\mathcal{J}), \mathbb{C}\text{-}\mathrm{vect}).$$
We have the projective resolution $P(\mathrm{Sd}(\mathcal{J}))_*$ of the constant functor $\underline{\mathbb{C}}$. Since $\mathrm{Sd}(\mathcal{J})$ is acyclic, so is  $\mathrm{Sd}(\mathcal{J})^{\mathrm{op}}$. Hence, we can apply Lemma \ref{AS}. Since the splitting functor preserves direct sums, for any object $\mathbf{f}$ of $\mathrm{Sd}(\mathcal{J})$ we obtain $$S_{\mathbf{f}}P(\mathrm{Sd}(\mathcal{J}))_*=\xymatrix{\dots\ar[r]&0\ar[r]&\mathbb{C}\ar[r]&0\ar[r]&\dots}$$ where $\mathbb{C}$ is only in the dimension $L(\mathbf{f})$ length $\mathbf{f}$. 
Since $\mathrm{Sd}(\mathcal{J})^{\mathrm{op}}$ is acyclic, $\mathrm{Aut}(\mathbf{f})$ is trivial, hence the tensor operation $-\bigotimes_{\mathbb{C}[\mathbf{f}]} \mathcal{N}(\mathbf{f})$ is trivial. Thus, we have
\begin{eqnarray*}
h^{(2)}\left( S_{\mathbf{f}} P(\mathrm{Sd}(\mathcal{J}))_*  \bigotimes_{\mathbb{C}[\mathbf{f}]} \mathcal{N}(\mathbf{f}) \right) &= &h^{(2)}\left( S_{\mathbf{f}} P(\mathrm{Sd}(\mathcal{J}))_*   \right) \\
&=&h^{(2)}\sum_{n\ge 0} \dim_{\mathcal{N}(\mathbf{f})} \left( S_{\mathbf{f}} P(\mathrm{Sd}(\mathcal{J}))_n \right)\\
&=&1.
\end{eqnarray*}
Note that $\dim_{\mathcal{N}(\mathbf{f})}$ is just the dimension as $\mathbb{C}$-vector spaces. We obtain
\begin{eqnarray}
h^{(2)}(\mathrm{Sd}(\mathcal{J})^{\mathrm{op}})&=&\sum_{\mathbf{f} \in \mathrm{Ob}(\mathrm{Sd}(\mathcal{J})^{\mathrm{op}})} h^{(2)}\left( S_{\mathbf{f}} P(\mathrm{Sd}(\mathcal{J}))_* \bigotimes_{\mathbb{C}[\mathbf{f}]} \mathcal{N}(\mathbf{f}) \right) \nonumber \\
&=&\sum_{\mathbf{f} \in \mathrm{Ob}(\mathrm{Sd}(\mathcal{J})^{\mathrm{op}})} 1 \nonumber \\
&=&\sum_{\mathbf{f} \in \mathrm{Ob}(\mathrm{Sd}(\mathcal{J}))} 1 \nonumber \\
&=& \sum_{n=0}^\infty \# \overline{N_n}(\mathcal{J}). \label{converges}
\end{eqnarray}
The series $(\ref{converges})$ converges if and only if each $\overline{N_n}(\mathcal{J})$ is finite and there exists sufficiently large $M$ such that $\overline{N_n}(\mathcal{J})=\emptyset$ for $n>M$. Thus, Lemma \ref{iff} proves the first claim.

If $\mathcal{J}$ is finite acyclic, the series $(\ref{converges})$ converges, hence $\mathrm{Sd}(\mathcal{J})^{\mathrm{op}}$ is of type $(L^2)$. We have
\begin{eqnarray*}
\chi^{(2)}(\mathrm{Sd}(\mathcal{J})^{\mathrm{op}})&=&\sum_{\mathbf{f} \in \mathrm{Ob}(\mathrm{Sd}(\mathcal{J})^{\mathrm{op}})} \chi^{(2)}\left( S_{\mathbf{f}} P(\mathrm{Sd}(\mathcal{J}))_* \bigotimes_{\mathbb{C}[\mathbf{f}]} \mathcal{N}(\mathbf{f}) \right)  \\
&=&\sum_{\mathbf{f} \in \mathrm{Ob}(\mathrm{Sd}(\mathcal{J})^{\mathrm{op}})} (-1)^{L(\mathbf{f})}  \\
&=& \sum_{n=0}^M (-1)^n\# \overline{N_n}(\mathcal{J}) \\
&=& \chi_L (\mathcal{J})
\end{eqnarray*}
for sufficiently large $M$. Lemma 7.3 of \cite{FLS} implies $\chi_L(\mathcal{J})=\chi^{(2)}(\mathcal{J})$. We conclude $$\chi^{(2)}(\mathrm{Sd}(\mathcal{J})^{\mathrm{op}})=\chi^{(2)}(\mathcal{J}).$$
\end{proof}

\subsection{The extended $L^2$-Euler characteristic}

In this subsection, we extend the definition of the $L^2$-Euler characteristic. As we have seen, the $L^2$-Euler characteristic is the invariant under the barycentric subdivision only for finite acyclic categories. We show the extended $L^2$-Euler characteristic is invariant under the barycentric subdivision for a wider class of finite categories, that is, the class for which the series Euler characteristic can be defined.


\begin{defi}
A small category $\mathcal{J}$ is called \textit{of type extended $(L^2)$} if for some projective resolution $P_*$ of the constant functor $\underline{\mathbb{C}}$ in $\mathrm{Func}(\mathcal{J}^{\mathrm{op}}, \mathbb{C}\text{-}\mathrm{vect}),$
$$h^{(2)}_n(\mathcal{J})=\sum_{[x]\in \text{iso($\mathcal{J})$}} h^{(2)} (S_x P_n \otimes_{\mathbb{C}[x]} \mathcal{N}(x)) $$
converges, the radius of convergence $\rho$ of the power series with complex variable
$$f^{(2)}_{\mathcal{J}}(z)=\sum^{\infty}_{n=0} h^{(2)}_n (\mathcal{J}) z^n $$
is not zero, there exist a real number $\varepsilon$ and a function $g$ such that
\begin{enumerate}
\item $\varepsilon \in (1,\infty]$
\item $g$ has finitely many poles except for $-1$ on $U(0;\varepsilon)$  with the center 0
\item $g$ is holomorphic in the open ball $U(0;\varepsilon)$ except for its poles
\item $g(z)\equiv f^{(2)}_{\mathcal{J}}(z)$ on $U(0;\varepsilon)$.
\end{enumerate}
Then, define \textit{the extended $L^2$-Euler characteristic} $\chi^{(2)}_{\mathrm{ex}}(\mathcal{J})$ of  $\mathcal{J}$ by
$$\chi^{(2)}_{\mathrm{ex}}(\mathcal{J})=g(-1).$$
\end{defi}

If there exist another $\varepsilon'$ and $g'$, then the uniqueness of the analytic continuity assures $g(z)\equiv g'(z)$ in $U(0;\min\{\varepsilon,\varepsilon' \})$. So this definition is well-defined. 

\begin{prop}
If a small category $\mathcal{J}$ is of type $(L^2)$, then $\mathcal{J}$ is of type extended $(L^2)$.
\begin{proof}
Since $\mathcal{J}$ is of type $(L^2)$, the series $h^{(2)}(\mathcal{J})$ converges absolutely. Hence, each $h_n^{(2)}(\mathcal{J})$ also converges absolutely. For a complex number $z_0$ such that $|z_0|=1$, we have
\begin{eqnarray*}
\sum^{\infty}_{n=0}|h^{(2)}_n z_0|&=&\sum^{\infty}_{n=0}|h^{(2)}_n|| z_0| \\
&=&\sum^{\infty}_{n=0} h^{(2)}_n <\infty
\end{eqnarray*}
Thus, $\sum^{\infty}_{n=0} h^{(2)}_n z_0$ converges absolutely. Hence, the radius of convergence $\rho$ of $f_{\mathcal{J}}^{(2)}(z)$ is lager than 1. Since $f_{\mathcal{J}}^{(2)}(z)$ is the power series, it is holomorphic on $U(0;\rho)$. Hence, $\mathcal{J}$ is of type extended $(L^2)$.
\end{proof}
\end{prop}

\begin{them}
Suppose $\mathcal{I}$ is a finite category. Then, there exists the extended $L^2$-Euler characteristic $\chi_{\mathrm{ex}}^{(2)} (\mathrm{Sd}(\mathcal{I})^{\mathrm{op}})$ of $\mathrm{Sd}(\mathcal{I})^{\mathrm{op}}$ if and only if its series Euler characteristic $\chi_{\sum}(\mathcal{I})$ exists, in which case, we obtain
$$\chi_{\sum}(\mathcal{I})=\chi_{\mathrm{ex}}^{(2)} (\mathrm{Sd}(\mathcal{I})^{\mathrm{op}}).$$
\begin{proof}
By the first half of the proof of Theorem \ref{moto-main}, we obtain
$$f^{(2)}_{\mathrm{Sd}(\mathcal{I})^{\mathrm{op}}}(z)=\sum^{\infty}_{n=0}\#\overline{N_n}(\mathcal{I}) z^n.$$
The number $\#\overline{N_n}(\mathcal{I}) $ can be expressed by using matrices, that is, $\#\overline{N_n}(\mathcal{I}) =\mathrm{sum}(Z_{\mathcal{I}}-E )$. Since entries of $(Z_{\mathcal{I}}-E )$ are natural numbers, we obtain
$$\#\overline{N_n}(\mathcal{I})=\mathrm{sum}\{(Z_{\mathcal{I}}-E )^n\}\le \{ \mathrm{sum}(Z_{\mathcal{I}}-E )\}^n .$$
Hence, we have
\begin{eqnarray}
\sum^{\infty}_{n=0} | \#\overline{N_n}(\mathcal{I}) z^n |&=&\sum^{\infty}_{n=0}  \#\overline{N_n}(\mathcal{I})| z^n | \nonumber \\
&\le & \sum^{\infty}_{n=0} \{ \mathrm{sum}(Z_{\mathcal{I}}-E )\}^n | z^n | \label{series}
\end{eqnarray}
For $0\le z_0 < \frac{1}{\mathrm{sum}(Z_{\mathcal{I}}-E )}$, the series $(\ref{series})$ converges, hence $f^{(2)}_{\mathrm{Sd}(\mathcal{I})^{\mathrm{op}}}(z_0)$ also converges. So the radius of convergence of $f^{(2)}_{\mathrm{Sd}(\mathcal{I})^{\mathrm{op}}} (t)$ is not zero.

By Theorem 2.2 of \cite{Leib}, it follows that $f^{(2)}_{\mathrm{Sd}(\mathcal{I})^{\mathrm{op}}}$ has the rational expression 
$$ f_{\mathrm{Sd}(\mathcal{I})^{\mathrm{op}}}^{(2)}(z) =\frac{\text{sum}(\text{adj}(E-(Z_{\mathcal{I}}-E)z ))}{\text{det}(E-(Z_{\mathcal{I}}-E)z )}.$$
$f^{(2)}_{\mathrm{Sd}(\mathcal{I})^{\mathrm{op}}}$ has finitely many poles on $U(0;\varepsilon)$. Hence, $\mathrm{Sd}(\mathcal{J})^{\mathrm{op}}$ is of type extended $(L^2)$ if and only if it does not have a pole at $-1$ and it is equivalent to the existence of $\chi_{\sum}(\mathcal{I})$. Finally, we obtain 
\begin{eqnarray*}
\chi^{(2)}_{\mathrm{ex}}(\mathrm{Sd}(\mathcal{J})^{\mathrm{op}})&=& \frac{\text{sum}(\text{adj}(E-(Z_{\mathcal{I}}-E)(-1) ))}{\text{det}(E-(Z_{\mathcal{I}}-E)(-1) )} \\
&=&\chi_{\sum}(\mathcal{I}).
\end{eqnarray*}
\end{proof}
\end{them}

\begin{rema}
We defined an extension of the $L^2$-Euler characteristic which turns out to be not invariant under equivalence of categories, since the series Euler characteristic is not. In Lemma 5.15 of \cite{FLS}, it was proven that the $L^2$-Euler characteristic is invariant under equivalence of categories for directly finite categories. 

\end{rema}

\end{document}